\documentclass[oneside,english]{amsart}
\usepackage{lmodern}
\usepackage[T1]{fontenc}
\usepackage[latin9]{inputenc}
\usepackage{color}
\usepackage{babel}
\usepackage{amsthm}
\usepackage{amssymb}
\usepackage[unicode=true,pdfusetitle,
 bookmarks=true,bookmarksnumbered=true,bookmarksopen=true,bookmarksopenlevel=1,
 breaklinks=false,pdfborder={0 0 0},backref=false,colorlinks=true]
 {hyperref}
\usepackage{breakurl}

\makeatletter
\numberwithin{equation}{section}
\numberwithin{figure}{section}
\theoremstyle{plain}
\newtheorem{thm}{\protect\theoremname}[section]
  \theoremstyle{definition}
  \newtheorem{defn}[thm]{\protect\definitionname}
  \theoremstyle{plain}
  \newtheorem{prop}[thm]{\protect\propositionname}
  \theoremstyle{remark}
  \newtheorem{rem}[thm]{\protect\remarkname}
  \theoremstyle{plain}
  \newtheorem{lem}[thm]{\protect\lemmaname}

\usepackage[a4paper]{geometry}
\allowdisplaybreaks[2]

\usepackage{stmaryrd}
\usepackage{amsfonts}
\usepackage{bbm}
\usepackage{amscd}
\usepackage{mathrsfs}
\usepackage{latexsym}
\usepackage{amscd}
\usepackage{amscd}
\usepackage{amsthm}
\usepackage{amsxtra}
\usepackage{xypic}
\usepackage{lmodern}
\usepackage[all]{xy}
\usepackage{nicefrac}
\usepackage{mathtools}
\usepackage{enumitem}
\usepackage{microtype}
\usepackage{pdfpages}
\usepackage{changepage}
\usepackage[toc,page,title,titletoc,header]{appendix} 

\newcommand{\be }{\begin{equation}}
\newcommand{\ee }{\end{equation}}

\newcommand {\emptycomment}[1]{} 

\newcommand{\br}[1]{   [ \cdot,    \cdot  ]   }

\makeatother

  \providecommand{\definitionname}{Definition}
  \providecommand{\lemmaname}{Lemma}
  \providecommand{\propositionname}{Proposition}
  \providecommand{\remarkname}{Remark}
\providecommand{\theoremname}{Theorem}

\begin{document}

\title{Derived brackets for fat Leibniz algebras}

\author{Xiongwei Cai and Zhangju Liu}

\address{Mathematics Research Unit, FSTC, University of Luxembourg, Luxembourg}

\email{shernvey@gmail.com}

\address{Department of Mathematics, Peking University, Beijing}

\email{liuzj@pku.edu.cn}

\keywords{Leibniz algebras, representable cochains, derived brackets}
\begin{abstract}
Given a Leibniz algebra $L$ with left center $Z$, we work on $C(L,Z,S^{\bullet}(Z))$,
the $Z$-standard complex of $L$ with coefficients in $S^{\bullet}(Z)$.
We construct the derived bracket for a fat Leibniz algebra in terms
of a certain 3-cocycle and a Poisson algebra structure on the space
of so-called ``representable cochains''.
\end{abstract}
\maketitle

\section{Introduction}

Leibniz algebras, objects that first appeared in Bloh's work \cite{Bloh}
and named by Loday \cite{Loday}, can be viewed as the noncommutative
analogue of Lie algebras. Some theorems and properties of Lie algebras
have been proved to be still valid for Leibniz algebras, while many
other questions are still open. 

Courant algebroids, first introduced by Liu, Weinstein and Xu in \cite{LiuWX},
can be viewed as the geometric realization of Leibniz algebras. The
algebraization of Courant algebroids, Courant-Dorfman algebras, are
special examples of Leibniz algebras.

The derived bracket for a Lie algebra with an ad-invariant inner product
is constructed by Lecomte-Roger \cite{LecomteRoger} and Kosmann-Schwarzbach
\cite{KS}, in order to study the homological algebra of Lie bialgebras
and quasi-Lie bialgebras, respectively. While the construction of
derived bracket for a Courant algebroid was given by Kosmann-Schwarzbach
\cite{KS04}, Royternberg \cite{Roytenberg02} and Alekseev-Xu \cite{AlekseevXu}.

It is a natural question to ask whether there is a derived bracket
construction for Leibniz algebras. In this paper, we succeed to give
a positive answer for fat Leibniz algebras. By a fat Leibniz algebra,
we mean a Leibniz algebra whose naturally defined symmetric product
is non-degenerate. Note that this is a different notion from a quadratic
Leibniz algebra, defined by Benayadi-Hidri \cite{BenayadiHidri}.

Given a Leibniz algebra $L$ with left center $Z$, we will work on
the $H$-standard complex (see Cai \cite{Cai}) of $L$ in the particular
case when $H=Z,\ V=S^{\bullet}(Z)$. We will define a canonical 3-cocycle
$\Theta$ and prove that the subcomplex consisting of the so-called
``representable cochains'' is a graded Poisson algebra. Finally
we show that the Leibniz bracket of a fat Leibniz algebra can be represented
by a derived bracket.

\subsection*{Acknowledgements}

This paper is based on the PhD dissertation of the first author, which
is funded by the University of Luxembourg. The first author would
like to thank his advisors, Prof. Martin Schlichenmaier and Prof.
Ping Xu, for their continual encouragement and support.

\section{Standard complex}

In this section, we recall the definition of $H$-standard complex
of a Leibniz algebra $L$ with coefficients in $V$ (\cite{Cai}),
and consider a 3-cocycle in the particular case when $H=Z,\ V=S^{\bullet}(Z)$.

Given a Leibniz algebra $L$ with left center $Z$, let $H\supseteq Z$
be an isotropic ideal in $L$, and $(V,\tau)$ be an $H$-trivial
representation of $L$ (i.e. a left representation of $L$ on which
$H$ acts trivially).

Denote by $C^{n}(L,H,V)$ the space of all sequences $\omega=(\omega_{0},\cdots,\omega_{[\frac{n}{2}]})$,
where $\omega_{k}$ is a linear map from $(\otimes^{n-2k}L)\otimes(\odot^{k}H)$
to $V$, $\forall k$, and is weakly skew-symmetric in arguments of
$L$ up to $\omega_{k+1}$:
\begin{eqnarray*}
 &  & \omega_{k}(e_{1},\cdots e_{i},e_{i+1,}\cdots e_{n-2k};h_{1},\cdots h_{k})+\omega_{k}(e_{1},\cdots e_{i+1},e_{i,}\cdots e_{n-2k};h_{1},\cdots h_{k})\\
 & = & -\omega_{k+1}(\cdots\widehat{e_{i}},\widehat{e_{i+1}},\cdots;(e_{i},e_{i+1}),\cdots)\qquad\forall e\in L,\ h\in H
\end{eqnarray*}
$C(L,H,V)\triangleq\bigoplus_{n}C^{n}(L,H,V)$ becomes a cochain complex
under the coboundary map $d=d_{0}+\delta+d^{\prime}$, called the
$H$-standard complex of $L$ with coefficients in $V$, where $d_{0},\delta,d^{\prime}$
are defined for any $\omega\in C^{n}(L,H,V)$ respectively by:
\begin{eqnarray*}
(d_{0}\omega)_{k}(e_{1},\cdots,e_{n+1-2k};h_{1},\cdots h_{k}) & \triangleq & \sum_{a}(-1)^{a+1}\rho(e_{a})\omega_{k}(\cdots\widehat{e_{a}},\cdots;\cdots)\\
 &  & +\sum_{a<b}(-1)^{a}\omega_{k}(\cdots\widehat{e_{a}},\cdots e_{a}\circ e_{b},\cdots;\cdots)\\
(\delta\omega)_{k}(e_{1},\cdots,e_{n+1-2k};h_{1},\cdots h_{k}) & \triangleq & \sum_{j}\omega_{k-1}(\alpha_{j},e_{1},\cdots e_{n+1-2k};\cdots\widehat{h_{j}},\cdots)\\
(d^{\prime}\omega)_{k}(e_{1},\cdots e_{n+1-2k};h_{1},\cdots h_{k}) & \triangleq & \sum_{a,j}(-1)^{a+1}\omega_{k}(\cdots\widehat{e_{a}},\cdots;\cdots\widehat{h_{j}},h_{j}\circ e_{a},\cdots).
\end{eqnarray*}

The Leibniz bracket of $L$ induces a left action $\rho$ of $L$
on $Z$: $\rho(e)f\triangleq e\circ f$, $\forall e\in L,\ f\in Z$.
And it can be extended by Leibniz rule to a left action of $L$ on
the symmetric tensor $S^{\bullet}(Z)$, still denoted by $\rho$.
$(S^{\bullet}(Z),\rho)$ is obviously a $Z$-trivial representation
of $L$, so we have the $Z$-standard complex $(C(L,Z,S^{\bullet}(Z)),d)$.
Note that $d^{\prime}$ is $0$, so $d=d_{0}+\delta$ in this case.
\begin{defn}
$(C(L,Z,S^{\bullet}(Z)),d=d_{0}+\delta)$ is called the standard complex
of $L$. 
\end{defn}
For simplicity, we will denote $C(L,Z,S^{\bullet}(Z))$ by $C(L)$
from now on.
\begin{prop}
$C(L)$ is a differential graded commutative algebra, with the multiplication
map defined for $\omega\in C^{n}(L),\ \eta\in C^{m}(L)$ by:
\begin{eqnarray}
 &  & (\omega\cdot\eta)_{k}(e_{1},\cdots,e_{n+m-2k};f_{1},\cdots,f_{k})\label{eq:multiplication}\\
 & \triangleq & \sum_{{i+j=k\atop {\sigma\in sh(n-2i,m-2j)\atop \mu\in sh(i,j)}}}(-1)^{\sigma}\omega_{i}(e_{\sigma(1)}\cdots e_{\sigma(n-2i)};f_{\mu(1)}\cdots f_{\mu(i)})\eta_{j}(e_{\sigma(n-2i+1)}\cdots;f_{\mu(i+1)}\cdots),\nonumber 
\end{eqnarray}
$\forall e\in L,f\in Z$, where $sh(\ ,\ )$ means the shuffle permutation.\end{prop}
\begin{proof}
The multiplication map above is obviously graded commutative, i.e.
\[
\omega\cdot\eta=(-1)^{nm}\eta\cdot\omega.
\]
We give the proof in 3 steps.

Step 1:

$C(L)$ is closed under the multiplication, i.e. $\omega\cdot\eta\in C^{n+m}(L)$:
\begin{eqnarray*}
 &  & (\omega\cdot\eta)_{k}(\cdots e_{a},e_{a+1},\cdots;f_{1},\cdots f_{k})+(\omega\cdot\eta)_{k}(\cdots e_{a+1},e_{a},\cdots;f_{1},\cdots f_{k})\\
 & = & \sum_{{i+j=k\atop \tau\in sh(i,j)}}\sum_{{\sigma\in sh(n-2i,m-2j)\atop \sigma^{-1}(a),\sigma^{-1}(a+1)\leq n-2i}}(-1)^{\sigma}\big(\omega_{i}(\cdots e_{a},e_{a+1},\cdots;\cdots)+\omega_{i}(\cdots e_{a+1},e_{a},\cdots;\cdots)\big)\eta_{j}(\cdots)\\
 &  & +\sum_{{i+j=k\atop \tau\in sh(i,j)}}\sum_{{\sigma\in sh(n-2i,m-2j)\atop \sigma^{-1}(a),\sigma^{-1}(a+1)>n-2i}}(-1)^{\sigma}\omega_{i}(\cdots)\big(\eta_{j}(\cdots e_{a},e_{a+1},\cdots;\cdots)+\eta_{i}(\cdots e_{a+1},e_{a},\cdots;\cdots)\big)\\
 &  & +\sum_{{i+j=k\atop \tau\in sh(i,j)}}\sum_{{\sigma\in sh(n-2i,m-2j)\atop \sigma^{-1}(a)\leq n-2i<\sigma^{-1}(a+1)}}(-1)^{\sigma}\big(\omega_{i}(\cdots e_{a}\cdots)\eta_{j}(\cdots e_{a+1}\cdots)+\omega_{i}(\cdots e_{a+1}\cdots)\eta_{j}(\cdots e_{a}\cdots)\big)\\
 &  & +\sum_{{i+j=k\atop \tau\in sh(i,j)}}\sum_{{\sigma\in sh(n-2i,m-2j)\atop \sigma^{-1}(a+1)\leq n-2i<\sigma^{-1}(a)}}(-1)^{\sigma}\big(\omega_{i}(\cdots e_{a+1}\cdots)\eta_{j}(\cdots e_{a}\cdots)+\omega_{i}(\cdots e_{a}\cdots)\eta_{j}(\cdots e_{a+1}\cdots)\big)\\
 &  & \mbox{(note\ that\ the\ same\ sequence\ \ensuremath{(\cdots e_{a},\cdots e_{a+1},\cdots)}\ viewed\ as\ permutations}\\
 &  & \mbox{of\ \ensuremath{(\cdots e_{a},e_{a+1},\cdots)}\ and\ \ensuremath{(\cdots e_{a+1},e_{a},\cdots)}\ have\ opposite\ signs)}\\
 & = & \sum_{{i+j=k\atop \tau\in sh(i,j)}}\sum_{{\sigma\in sh(n-2i,m-2j)\atop \sigma^{-1}(a),\sigma^{-1}(a+1)\leq n-2i}}(-1)^{\sigma+1}\omega_{i+1}(\cdots,\widehat{e_{a}},\widehat{e_{a+1}},\cdots;(e_{a},e_{a+1}),\cdots)\eta_{j}(\cdots)\\
 &  & +\sum_{{i+j=k\atop \tau\in sh(i,j)}}\sum_{{\sigma\in sh(n-2i,m-2j)\atop \sigma^{-1}(a),\sigma^{-1}(a+1)>n-2i}}(-1)^{\sigma+1}\omega_{i}(\cdots)\eta_{j+1}(\cdots,\widehat{e_{a}},\widehat{e_{a+1}},\cdots;(e_{a},e_{a+1}),\cdots)\\
 & = & \sum_{{l+j=k+1\atop \sigma\in sh(n-2l,m-2j)}}\sum_{{\tau\in sh(l,j)\atop \tau^{-1}((e_{a},e_{a+1}))\leq l}}(-1)^{\sigma+1}\omega_{l}(\cdots;(e_{a},e_{a+1}),\cdots)\eta_{j}(\cdots)\\
 &  & +\sum_{{i+l=k+1\atop \sigma\in sh(n-2i,m-2l)}}\sum_{{\tau\in sh(i,l)\atop \tau^{-1}((e_{a},e_{a+1}))>i}}(-1)^{\sigma+1}\omega_{i}(\cdots)\eta_{l}(\cdots;(e_{a},e_{a+1}),\cdots)\\
 & = & -(\omega\cdot\eta)_{k+1}(e_{1},\cdots,\widehat{e_{a}},\widehat{e_{a+1}},\cdots;(e_{a},e_{a+1}),\cdots)
\end{eqnarray*}

Step 2:

The multiplication is associative:

$\forall\omega\in C^{n}(L),\ \eta\in C^{m}(L),\ \lambda\in C^{l}(L)$,
by definition it is an easy calculation that, $((\omega\cdot\eta)\cdot\lambda)_{k}(e_{1},\cdots,e_{n+m+l-2k};f_{1},\cdots,f_{k})$
and $(\omega\cdot(\eta\cdot\lambda))_{k}(e_{1},\cdots,e_{n+m+l-2k};f_{1},\cdots,f_{k})$
both equal to:
\[
\sum_{{a+b+c=k\atop {\sigma\in sh(n-2a,m-2b,l-2c)\atop \tau\in sh(a,b,c)}}}(-1)^{\sigma}\omega_{a}(\cdots)\eta_{b}(\cdots)\lambda_{c}(\cdots)
\]

Step 3:

The differential $d$ is a graded derivation:
\[
d(\omega\cdot\eta)=(d\omega)\cdot\eta+(-1)^{n}\omega\cdot(d\eta),\ \forall\omega\in C^{n}(L),\ \eta\in C^{m}(L).
\]

Since $d=d_{0}+\delta$, it suffices to prove the equation for $d_{0},\delta$
respectively.

For $d_{0}$, we only give the proof for the case of degree $0$ here,
since the proof is almost the same for cases of higher degrees (the
only difference is that the sum should be taken over permutations
of the arguments in $Z$ as well).

\begin{eqnarray*}
 &  & (d_{0}(\omega\cdot\eta))_{0}(e_{1},\cdots,e_{n+m+1})\\
 & = & \sum_{a}(-1)^{a+1}\rho(e_{a})(\omega\cdot\eta)_{0}(\cdots\widehat{e_{a}}\cdots)+\sum_{a<b}(-1)^{a}(\omega\cdot\eta)_{0}(\cdots\widehat{e_{a}}\cdots\widehat{e_{b}},e_{a}\circ e_{b}\cdots)\\
 & = & \sum_{a}(-1)^{a+1}\rho(e_{a})\big(\sum_{\sigma\in sh(n,m)\{\cdots,\hat{a},\cdots\}}(-1)^{\sigma}\omega_{0}(e_{\sigma(1)}\cdots e_{\sigma(n)})\eta_{0}(e_{\sigma(n+1)}\cdots e_{\sigma(n+m)})\big)\\
 &  & +\sum_{a<b}(-1)^{a}\sum_{{\sigma\in sh(n,m)\{\cdots\hat{a},\cdots\}\atop \sigma^{-1}(b)<n+1}}(-1)^{\sigma}\omega_{0}(e_{\sigma(1)},\cdots\widehat{e_{b}},e_{a}\circ e_{b},\cdots e_{\sigma(n)})\eta_{0}(e_{\sigma(n+1)},\cdots e_{\sigma(n+m+1)})\\
 &  & +\sum_{a<b}(-1)^{a}\sum_{{\sigma\in sh(n,m)\{\cdots\hat{a},\cdots\}\atop \sigma^{-1}(b)>n}}(-1)^{\sigma}\omega_{0}(e_{\sigma(1)},\cdots e_{\sigma(n)})\eta_{0}(e_{\sigma(n+1)},\cdots\widehat{e_{b}},e_{a}\circ e_{b},\cdots e_{\sigma(n+m+1)})\\
 &  & \mbox{(let\ \ensuremath{\sigma_{1}}, \ensuremath{\sigma_{2}}\ be\ the\ permutations\ adding\ \ensuremath{a}\ to\ \ensuremath{\sigma}\ in\ front and at back respectively)}\\
 & = & \sum_{a}\sum_{\sigma_{1}\in sh(n+1,m)}(-1)^{a+1}(-1)^{\sigma_{1}+\sigma_{1}^{-1}(a)-a}\\
 &  & \quad\big(\rho(e_{a})\omega_{0}(e_{\sigma_{1}(1)}\cdots\widehat{e_{a}},e_{\sigma_{1}(\sigma_{1}^{-1}(a)+1)}\cdots e_{\sigma_{1}(n+1)})\big)\eta_{0}(e_{\sigma_{1}(n+2)}\cdots e_{\sigma_{1}(n+m+1)})\\
 &  & +\sum_{a}\sum_{\sigma_{2}\in sh(n,m+1)}(-1)^{a+1}(-1)^{\sigma_{2}+\sigma_{2}^{-1}(a)-a}\\
 &  & \quad\omega_{0}(e_{\sigma_{2}(1)}\cdots e_{\sigma_{2}(n)})\big(\rho(e_{a})\eta_{0}(e_{\sigma_{2}(n+1)}\cdots\widehat{e_{a}},e_{\sigma_{2}(\sigma_{2}^{-1}(a)+1)}\cdots e_{\sigma_{2}(n+m+1)})\big)\\
 &  & +\sum_{a<b}\sum_{{\sigma_{1}\in sh(n+1,m)\atop \sigma_{1}^{-1}(b)<n+2}}(-1)^{a}(-1)^{\sigma_{1}+\sigma_{1}^{-1}(a)-a}\\
 &  & \quad\omega_{0}(e_{\sigma_{1}(1)}\cdots\widehat{e_{a}},e_{\sigma_{1}(\sigma_{1}^{-1}(a)+1)}\cdots\widehat{e_{b}},e_{a}\circ e_{b}\cdots)\eta_{0}(e_{\sigma_{1}(n+2)}\cdots e_{\sigma_{1}(n+m+1)})\\
 &  & +\sum_{a<b}\sum_{{\sigma_{2}\in sh(n,m+1)\atop \sigma_{2}^{-1}(b)>n+1}}(-1)^{a}(-1)^{\sigma_{2}+\sigma_{2}^{-1}(a)-a}\\
 &  & \quad\omega_{0}(e_{\sigma_{2}(1)}\cdots)\eta_{0}(e_{\sigma_{2}(n+1)}\cdots\widehat{e_{a}},e_{\sigma_{2}(\sigma_{2}^{-1}(a)+1)}\cdots\widehat{e_{b}},e_{a}\circ e_{b}\cdots e_{\sigma_{2}(n+m+1)})\\
 & = & \sum_{\sigma_{1}}(-1)^{\sigma_{1}}\sum_{(a_{1}\triangleq\sigma_{1}^{-1}(a))<n+2}(-1)^{a_{1}+1}\\
 &  & \quad\big(\rho(e_{\sigma_{1}(a_{1})})\omega_{0}(e_{\sigma_{1}(1)}\cdots\widehat{e_{\sigma_{1}(a_{1})}}\cdots e_{\sigma_{1}(n+1)})\big)\eta_{0}(e_{\sigma_{1}(n+2)}\cdots e_{\sigma_{1}(n+m+1)})\\
 &  & +\sum_{\sigma_{1}}(-1)^{\sigma_{1}}\sum_{(a_{1}\triangleq\sigma_{1}^{-1}(a))<(b_{1}\triangleq\sigma_{1}^{-1}(b))<n+2}(-1)^{a_{1}}\\
 &  & \quad\omega_{0}(e_{\sigma_{1}(1)}\cdots\widehat{e_{\sigma_{1}(a_{1})}}\cdots\widehat{e_{\sigma_{1}(b_{1})}},e_{\sigma_{1}(a_{1})}\circ e_{\sigma_{1}(b_{1})}\cdots)\eta_{0}(e_{\sigma_{1}(n+2)}\cdots e_{\sigma_{1}(n+m+1)})\\
 &  & +\sum_{\sigma_{2}}(-1)^{\sigma_{2}+n}\omega_{0}(e_{\sigma_{2}(1)},\cdots,e_{\sigma_{2}(n)})\cdot\sum_{a_{2}\triangleq\sigma_{2}^{-1}(a)}\\
 &  & \quad(-1)^{a_{2}-n+1}\rho(e_{\sigma_{2}(a_{2})})\eta_{0}(e_{\sigma_{2}(n+1)},\cdots,\widehat{e_{\sigma_{2}(a_{2})}},\cdots,e_{\sigma_{2}(n+m+1)})\\
 &  & +\sum_{\sigma_{2}}(-1)^{\sigma_{2}+n}\omega_{0}(e_{\sigma_{2}(1)},\cdots,e_{\sigma_{2}(n)})\cdot\sum_{n<(a_{2}\triangleq\sigma_{2}^{-1}(a))<(b_{2}\triangleq\sigma_{2}^{-1}(b))}\\
 &  & \quad(-1)^{a_{2}-n}\eta_{0}(e_{\sigma_{2}(n+1)}\cdots\widehat{e_{\sigma_{2}(a_{2})}}\cdots\widehat{e_{\sigma_{2}(b_{2})}},e_{\sigma_{2}(a_{2})}\circ e_{\sigma_{2}(b_{2})}\cdots e_{\sigma_{2}(n+m+1)})\\
 & = & \big((d_{0}\omega)\cdot\eta+(-1)^{n}\omega\cdot(d_{0}\eta)\big)_{0}(e_{1},\cdots,e_{n+m+1})
\end{eqnarray*}

For $\delta$,
\begin{eqnarray*}
 &  & (\delta(\omega\cdot\eta))_{k}(e_{1},\cdots,e_{n+m+1-2k};f_{1},\cdots,f_{k})\\
 & = & \sum_{i}(\omega\cdot\eta)_{k-1}(f_{i},e_{1},\cdots,e_{n+m+1-2k};\cdots,\hat{f_{i}},\cdots)\\
 & = & \sum_{i}\sum_{{a+b=k-1\atop \tau\in sh(a,b)}}\sum_{{\sigma\in sh(n-2a,m-2b)\atop \sigma^{-1}(f_{i})\leq n-2a}}(-1)^{\sigma}\omega_{a}(f_{i},\cdots;\cdots\hat{f_{i}}\cdots)\eta_{b}(\cdots)\\
 &  & +\sum_{i}\sum_{{a+b=k-1\atop \tau\in sh(a,b)}}\sum_{{\sigma\in sh(n-2a,m-2b)\atop \sigma^{-1}(f_{i})>n-2a}}(-1)^{\sigma}\omega_{a}(\cdots)\eta_{b}(f_{i},\cdots;\cdots\hat{f_{i}}\cdots)\\
 &  & \mbox{(removing\ \ensuremath{f_{i}}\ from\ \ensuremath{\sigma},\ adding\ \ensuremath{f_{i}}\ to\ \ensuremath{\tau}\ in\ front\ and\ at\ back\ respectively)}\\
 & = & \sum_{{a+b=k\atop \sigma\in sh(n+1-2a,m-2b)}}\sum_{{\tau\in sh(a,b)\atop \tau^{-1}(i)\leq a}}(-1)^{\sigma}\omega_{a-1}(f_{i},e_{\sigma(1)},\cdots;\cdots,\widehat{f_{\tau(\tau^{-1}(i))}},\cdots)\eta_{b}(\cdots)\\
 &  & +\sum_{{a+b=k\atop \sigma\in sh(n+1-2a,m-2b)}}\sum_{{\tau\in sh(a,b)\atop \tau^{-1}(i)>a}}(-1)^{\sigma+n}\omega_{a}(\cdots)\eta_{b-1}(f_{i},e_{\sigma(n-2a+1)},\cdots;\cdots,\widehat{f_{\tau(\tau^{-1}(i))}},\cdots)\\
 & = & ((\delta\omega)\cdot\eta)_{k}(\cdots)+(-1)^{n}(\omega\cdot(\delta\eta))_{k}(\cdots)
\end{eqnarray*}

The proof is finished.\end{proof}
\begin{rem}
\label{Remark:standard complex}In \cite{Cai}, we construct a Courant-Dorfman
algebra structure on  $S^{\bullet}(Z)\otimes L$ for any Leibniz algebra
$L$ with left center $Z$, and prove an isomorphism between $H$-standard
complexes of them. So it is a direct conclusion that $C(L)$ is isomorphic
to the standard complex of the Courant-Dorfman algebra $S^{\bullet}(Z)\otimes L$.
\end{rem}
Next we consider a 3-cochain in $C(L)$. Let $\Theta_{0}:L\otimes L\otimes L\rightarrow S^{\bullet}(Z)$
and $\Theta_{1}:L\otimes Z\rightarrow S^{\bullet}(Z)$ be defined
as:
\[
\Theta_{0}(e_{1},e_{2},e_{3})=(e_{1}\circ e_{2},e_{3})
\]
\begin{equation}
\Theta_{1}(e;f)=-(e,f).\label{eq: Theta}
\end{equation}

We can prove that $\Theta=(\Theta_{0},\Theta_{1})$ is a 3-cocycle
by a direct calculation, but actually we have the following:
\begin{prop}
$\Theta=d\zeta$ is a 3-coboundary, where $\zeta=(\zeta_{0},\zeta_{1})\in C^{2}(L)$
is defined by:
\[
\zeta_{0}(e_{1},e_{2})\triangleq(e_{1},e_{2}),\quad\zeta_{1}(f)\triangleq-2f.
\]
\end{prop}
\begin{proof}
Since 
\[
\zeta_{0}(e_{1},e_{2})+\zeta_{0}(e_{2},e_{1})=2(e_{1},e_{2})=-\zeta_{1}((e_{1},e_{2})),
\]
$\zeta=(\zeta_{0},\zeta_{1})$ is a 2-cochain in $C^{2}(L)$. By definition,
\begin{eqnarray*}
 &  & (d\zeta)_{0}(e_{1},e_{2},e_{3})\\
 & = & \rho(e_{1})\zeta_{0}(e_{2},e_{3})-\rho(e_{2})\zeta_{0}(e_{1},e_{3})+\rho(e_{3})\zeta_{0}(e_{1},e_{2})\\
 &  & -\zeta_{0}(e_{1}\circ e_{2},e_{3})-\zeta_{0}(e_{2},e_{1}\circ e_{3})+\zeta_{0}(e_{1},e_{2}\circ e_{3})\\
 & = & \rho(e_{1})(e_{2},e_{3})-(e_{1}\circ e_{2},e_{3})-(e_{2},e_{1}\circ e_{3})\\
 &  & -\rho(e_{2})(e_{1},e_{3})+(e_{1},e_{2}\circ e_{3})+\rho(e_{3})(e_{1},e_{2})\\
 & = & -(e_{2}\circ e_{1},e_{3})+\rho(e_{3})\zeta_{0}(e_{1},e_{2})\\
 & = & (e_{1}\circ e_{2},e_{3})
\end{eqnarray*}
\[
(d\zeta)_{1}(e;f)=\rho(e)\zeta_{1}(f)+\zeta_{0}(f,e)=-2(e,f)+(e,f)=-(e,f)
\]
so $\Theta=d\zeta$ is a 3-coboundary.
\end{proof}
Actually $\Theta$ is exactly the restriction of the canonical 3-cocycle
of the Courant-Dorfman algebra $S^{\bullet}(Z)\otimes L$ (see Remark
\ref{Remark:standard complex}). We will call $\Theta$ the canonical
3-cocycle of $L$.

\section{Poisson structure on a subcomplex}

In this section, we consider a subcomplex, denoted by $\tilde{C}(L)$,
consisting of the so-called ``representable cochains'', and construct
a Poisson algebra structure on $\tilde{C}(L)$.

Let $L^{\vee}\triangleq Hom(L,\ S^{\bullet}(Z))$. In this paper,
$Hom$ always means $\mathfrak{k}$-linear homomorphisms.

$\forall\omega\in C^{n}(L)$, $\omega_{k}$ gives rise to a map $\bar{\omega}_{k}:\ L^{\otimes n-2k-1}\otimes S^{k}(Z)\rightarrow L^{\vee}:$
\[
\bar{\omega}_{k}(e_{1},\cdots,e_{n-2k-1};f_{1},\cdots f_{k})(e)\triangleq(\iota_{f_{k}}\cdots\iota_{f_{1}}\iota_{e_{n-2k-1}}\cdots\iota_{e_{1}}\omega_{k})(e)=\omega_{k}(e_{1},\cdots,e_{n-2k-1},e;f_{1},\cdots,f_{k}).
\]

The symmetric product $(\cdot,\cdot)$ of $L$ can be $S^{\bullet}(Z)$-linearly
extended to a symmetric product on $S^{\bullet}(Z)\otimes L$, thus
inducing a map 
\[
\phi\triangleq(\cdot,\ ):\ S^{\bullet}(Z)\otimes L\rightarrow L^{\vee}.
\]

\begin{defn}
\label{Def:good cochains}Given any $\omega\in C^{n}(L)$, if $Im(\bar{\omega}_{k})\subseteq Im(\phi),\ \forall k$,
we call $\omega$ a ``representable cochain''. The graded subspace
of $C(L)$ consisting of all representable cochains is denoted by
$\tilde{C}(L)$.
\end{defn}
By definition, $e^{\flat}\triangleq(e,\ ):\ L\rightarrow Z$ is obviously
a representable cochain.

Given $\omega\in\tilde{C}^{n}(L)$, $\omega_{k}$ induces a $\mathfrak{k}$-linear
map 
\[
\tilde{\omega}_{k}:\ L^{\otimes n-2k-1}\rightarrow Hom(S^{k}(Z),\ S^{\bullet}(Z)\otimes L),
\]
which is defined by
\[
\tilde{\omega}_{k}(e_{1},\cdots,e_{n-2k-1})(f_{1},\cdots,f_{k})\triangleq\phi^{-1}(\bar{\omega}_{k}(e_{1},\cdots,e_{n-2k-1};f_{1},\cdots,f_{k})).
\]
 Note that, to determine $\tilde{\omega}_{k}$, we only need to choose
the preimage of $\bar{\omega}_{k}$ for given basis of $L$ and $Z$,
and then take the $\mathfrak{k}$-linear extension. So $\tilde{\omega}_{k}$
depends on the choices, it is not uniquely determined unless $\phi$
is injective (i.e. the bilinear product of $L$ is non-degenerate).
\begin{prop}
$\tilde{C}(L)$ is a subcomplex of $C(L)$.\end{prop}
\begin{proof}
$\forall\omega\in\tilde{C}^{n}(L),$ we need to prove that $d\omega\in\tilde{C}^{n+1}(L)$:
\begin{eqnarray*}
 &  & (d\omega)_{k}(e_{1},\cdots e_{n+1-2k};f_{1},\cdots f_{k})\\
 & = & \sum_{a}(-1)^{a+1}\rho(e_{a})\omega_{k}(\cdots\widehat{e_{a}},\cdots;\cdots)+\sum_{a<b}(-1)^{a}\omega_{k}(\cdots\widehat{e_{a}},\cdots e_{a}\circ e_{b},\cdots;\cdots)\\
 &  & +\sum_{i}\omega_{k-1}(f_{i},e_{1},\cdots;\cdots\widehat{f_{i}},\cdots)\\
 & = & \sum_{a\leq n-2k}(-1)^{a+1}\rho(e_{a})(\tilde{\omega}_{k}(e_{1},\cdots\widehat{e_{a}},\cdots e_{n-2k})(f_{1},\cdots f_{k}),e_{n+1-2k})\\
 &  & +(-1)^{n}\rho(e_{n+1-2k})\omega_{k}(e_{1},\cdots e_{n-2k};\cdots)\\
 &  & +\sum_{a<b\leq n-2k}(-1)^{a}(\tilde{\omega}_{k}(\cdots\widehat{e_{a}},\cdots e_{a}\circ e_{b},\cdots e_{n-2k})(f_{1},\cdots f_{k}),e_{n+1-2k})\\
 &  & +\sum_{a\leq n-2k}(-1)^{a}\omega_{k}(\cdots\widehat{e_{a}},\cdots e_{n-2k},e_{a}\circ e_{n+1-2k};\cdots)\\
 &  & +\sum_{i}(\tilde{\omega}_{k-1}(f_{i},e_{1},\cdots e_{n-2k})(\cdots\widehat{f_{i}},\cdots),e_{n+1-2k})\\
 & = & (\bullet,e_{n+1-2k}).
\end{eqnarray*}
The proof is finished.
\end{proof}
Next, we will define a graded bracket on $\tilde{C}(L)$.

$\forall\alpha\in Hom(S^{k}(Z),\ S^{\bullet}(Z)\otimes L),\ \beta\in Hom(S^{l}(Z),\ S^{\bullet}(Z)\otimes L)$,
define $\langle\alpha\cdot\beta\rangle\in Hom(S^{k+l}(Z),\ S^{\bullet}(Z))$
as
\[
\langle\alpha\cdot\beta\rangle(f_{1},\cdots,f_{k+l})\triangleq\sum_{\sigma\in sh(k,l)}(\alpha(f_{\sigma(1)},\cdots,f_{\sigma(k)}),\ \beta(f_{\sigma(k+1)},\cdots,f_{\sigma(k+l)})).
\]
 $\forall\gamma\in Hom(S^{k}(Z),\ S^{\bullet}(Z)),\ \delta\in Hom(S^{l}(Z),\ S^{\bullet}(Z))$,
define $\gamma\circ\delta\in Hom(S^{k+l-1}(Z),\ S^{\bullet}(Z))$
as
\[
\gamma\circ\delta(f_{1},\cdots,f_{k+l-1})\triangleq\sum_{\sigma\in sh(l,k-1)}\check{\gamma}(\delta(f_{\sigma(1)},\cdots,f_{\sigma(l)}),f_{\sigma(l+1)},\cdots,f_{\sigma(l+k-1)}),
\]

where $\check{\gamma}:\ S^{\bullet}(Z)\otimes S^{k-1}(Z)\rightarrow S^{\bullet}(Z)$
is extended from $\gamma$ by Leibniz rule in the first argument. 

Now given $\omega\in\tilde{C}^{n}(L),\ \eta\in\tilde{C}^{m}(L)$,
we define the bracket $\{\omega,\eta\}$ as follows:
\begin{equation}
\{\omega,\eta\}\triangleq\omega\bullet\eta+\omega\diamond\eta-(-1)^{nm}\eta\diamond\omega,\label{eq:Poisson bracket Leibniz}
\end{equation}

where $\omega\bullet\eta=((\omega\bullet\eta)_{0},(\omega\bullet\eta)_{1},\cdots)$,
with $(\omega\bullet\eta)_{k}:\ \otimes^{n+m-2-2k}L\rightarrow Hom(S^{k}(Z),\ S^{\bullet}(Z))$
defined by
\begin{eqnarray*}
 &  & (\omega\bullet\eta)_{k}(e_{1},\cdots,e_{n+m-2-2k})\\
 & \triangleq & (-1)^{m-1}\sum_{{i+j=k\atop \sigma\in sh(n-2i-1,m-2j-1)}}(-1)^{\sigma}\langle\tilde{\omega}_{i}(e_{\sigma(1)},\cdots e_{\sigma(n-2i-1)})\cdot\tilde{\eta}_{j}(e_{\sigma(n-2i)},\cdots e_{\sigma(n+m-2-2k)})\rangle,
\end{eqnarray*}

(obviously the value does not depend on the choices of $\tilde{\omega}_{i}$
and $\tilde{\eta}_{j}$, so it is well-defined) 

and $\omega\diamond\eta=((\omega\diamond\eta)_{0},(\omega\diamond\eta)_{1},\cdots)$,
with $(\omega\diamond\eta)_{k}:\ \otimes^{n+m-2-2k}L\rightarrow Hom(S^{k}(Z),\ S^{\bullet}(Z))$
defined by
\begin{eqnarray*}
 &  & (\omega\diamond\eta)_{k}(e_{1},\cdots,e_{n+m-2-2k})\\
 & \triangleq & \sum_{{i+j=k\atop \sigma\in sh(n-2i-2,m-2j)}}(-1)^{\sigma}\omega_{i+1}(e_{\sigma(1)}\cdots e_{\sigma(n-2i-2)})\circ\eta_{j}(e_{\sigma(n-2i-1)}\cdots e_{\sigma(n+m-2-2k)}).
\end{eqnarray*}
The following is the main theorem of this section:
\begin{thm}
$(\tilde{C}(L),\{\cdot,\cdot\})$ is a graded Poisson algebra. 

\label{thm:poisson bracket Leibniz}
\end{thm}
Before the proof of this theorem, we prove the following two lemmas
first.
\begin{lem}
$\tilde{C}(L)$ is a subalgebra of $C(L)$ with the multiplication
map defined in \ref{eq:multiplication}.\end{lem}
\begin{proof}
Given $\eta\in\tilde{C}^{m}(L),\lambda\in\tilde{C}^{l}(L)$, we need
to prove that $\eta\lambda\in\tilde{C}^{m+l}(L)$:
\begin{eqnarray*}
 &  & (\eta\lambda)_{k}(e_{1},\cdots,e_{m+l-2k};f_{1},\cdots,f_{k})\\
 & = & \sum_{{i+j=k\atop {\sigma\in sh(m-2i,l-2j)\atop \tau\in sh(i,j)}}}(-1)^{\sigma}\eta_{i}(e_{\sigma(1)}\cdots;f_{\tau(1)}\cdots f_{\tau(i)})\lambda_{j}(e_{\sigma(m-2i+1)}\cdots e_{\sigma(m+l-2k)};f_{\tau(i+1)}\cdots f_{\tau(k)})\\
 & = & \sum_{{i+j=k\atop \tau\in sh(i,j)}}\sum_{{\sigma\in sh(m-2i,l-2j)\atop \sigma^{-1}(m+l-2k)=m-2i}}(-1)^{\sigma}(\tilde{\eta_{i}}(\cdots,\widehat{e_{m+l-2k}};\cdots),e_{m+l-2k})\lambda_{j}(\cdots;\cdots)\\
 &  & +\sum_{{i+j=k\atop \tau\in sh(i,j)}}\sum_{{\sigma\in sh(m-2i,l-2j)\atop \sigma^{-1}(m+l-2k)=m+l-2k}}(-1)^{\sigma}\eta_{i}(\cdots;\cdots)(\tilde{\lambda_{j}}(\cdots,\widehat{e_{m+l-2k}};\cdots),e_{m+l-2k})\\
 & = & (\{\sum_{{i+j=k\atop \tau\in sh(i,j)}}\sum_{\bar{\sigma}\in sh(m-2i-1,l-2j)}(-1)^{\bar{\sigma}+l}\tilde{\eta_{i}}(e_{\bar{\sigma}(1)},\cdots;f_{\tau(1)},\cdots)\lambda_{j}(e_{\bar{\sigma}(m-2i)}\cdots;f_{\tau(i+1)},\cdots)\\
 &  & +\sum_{{i+j=k\atop \tau\in sh(i,j)}}\sum_{\bar{\sigma}\in sh(m-2i,l-2j-1)}(-1)^{\bar{\sigma}}\eta_{i}(e_{\bar{\sigma}(1)}\cdots;f_{\tau(1)},\cdots)\tilde{\lambda_{j}}(e_{\bar{\sigma}(m-2i+1)}\cdots;f_{\tau(i+1)},\cdots)\},e_{m+l-2k})
\end{eqnarray*}

The lemma is proved.\end{proof}
\begin{lem}
$\omega\bullet\eta,\ \omega\diamond\eta,\ \{\omega,\eta\}$ are all
cochains in $C^{n+m-2}(L)$.\end{lem}
\begin{proof}
$\omega\bullet\eta$ is a cochain in $C^{n+m-2}(L)$ because: 
\begin{eqnarray*}
 &  & (\omega\bullet\eta)_{k}(e_{1}\cdots e_{a},e_{a+1}\cdots e_{n+m-2-2k};f_{1}\cdots f_{k})+(\omega\bullet\eta)_{k}(\cdots e_{a+1},e_{a}\cdots;\cdots)\\
 & = & (-1)^{m-1}\sum_{i+j=k}\{\sum_{\sigma,a\in\omega,a+1\in\eta}(-1)^{\sigma}\langle\tilde{\omega_{i}}(e_{\sigma(1)}\cdots e_{a}\cdots)\cdot\tilde{\eta_{j}}(e_{\sigma(n-2i)}\cdots e_{a+1}\cdots e_{\sigma(n+m-2-2k)})\rangle\\
 &  & \quad+\sum_{\sigma,a\in\eta,a+1\in\omega}(-1)^{\sigma}\langle\tilde{\omega_{i}}(e_{\sigma(1)}\cdots e_{a}\cdots)\cdot\tilde{\eta_{j}}(e_{\sigma(n-2i)}\cdots e_{a+1}\cdots e_{\sigma(n+m-2-2k)})\rangle\}\\
 &  & +(-1)^{m-1}\sum_{i+j=k}\{\sum_{\sigma,a\in\eta,a+1\in\omega}(-1)^{\sigma}\langle\tilde{\omega_{i}}(e_{\sigma(1)}\cdots e_{a+1},\cdots)\cdot\tilde{\eta_{j}}(e_{\sigma(n-2i)}\cdots,e_{a}\cdots e_{\sigma(n+m-2-2k)})\rangle\\
 &  & \quad+\sum_{\sigma,a\in\omega,a+1\in\eta}(-1)^{\sigma}\langle\tilde{\omega_{i}}(e_{\sigma(1)}\cdots e_{a+1}\cdots)\cdot\tilde{\eta_{j}}(e_{\sigma(n-2i)}\cdots e_{a}\cdots e_{\sigma(n+m-2-2k)})\rangle\}\\
 &  & +(-1)^{m-1}\sum_{i+j=k}\{\sum_{\sigma,a\in\omega,a+1\in\omega}(-1)^{\sigma}\langle\tilde{\omega_{i}}(e_{\sigma(1)}\cdots e_{a},e_{a+1}\cdots)\cdot\tilde{\eta_{j}}(e_{\sigma(n-2i)}\cdots e_{\sigma(n+m-2-2k)})\rangle\\
 &  & \quad+\sum_{\sigma,a\in\omega,a+1\in\omega}(-1)^{\sigma}\langle\tilde{\omega_{i}}(e_{\sigma(1)}\cdots e_{a+1},e_{a}\cdots)\cdot\tilde{\eta_{j}}(e_{\sigma(n-2i)}\cdots e_{\sigma(n+m-2-2k)})\rangle\}\\
 &  & +(-1)^{m-1}\sum_{i+j=k}\{\sum_{\sigma,a\in\eta,a+1\in\eta}(-1)^{\sigma}\langle\tilde{\omega_{i}}(e_{\sigma(1)}\cdots)\cdot\tilde{\eta_{j}}(e_{\sigma(n-2i)}\cdots e_{a},e_{a+1}\cdots e_{\sigma(n+m-2-2k)})\rangle\\
 &  & \quad+\sum_{\sigma,a\in\eta,a+1\in\eta}(-1)^{\sigma}\langle\tilde{\omega_{i}}(e_{\sigma(1)}\cdots)\cdot\tilde{\eta_{j}}(e_{\sigma(n-2i)}\cdots e_{a+1},e_{a}\cdots)\rangle\}\\
 & = & (-1)^{m-1}\sum_{i+j=k}\sum_{{\sigma\in sh(n-2i-1,m-2j-1)\atop a\in\omega,a+1\in\omega}}(-1)^{\sigma}(-1)\\
 &  & \quad\langle\tilde{\omega_{i+1}}(e_{\sigma(1)}\cdots\widehat{e_{a}},\widehat{e_{a+1}}\cdots e_{\sigma(n-2i-1)};(e_{a},e_{a+1}))\cdot\tilde{\eta_{j}}(e_{\sigma(n-2i)}\cdots)\rangle\\
 &  & +(-1)^{m-1}\sum_{i+j=k}\sum_{{\sigma\in sh(n-2i-1,m-2j-1)\atop a\in\omega,a+1\in\omega}}(-1)^{\sigma}(-1)\\
 &  & \quad\langle\tilde{\omega_{i}}(e_{\sigma(1)}\cdots)\cdot\tilde{\eta_{j+1}}(e_{\sigma(n-2i)}\cdots\widehat{e_{a}},\widehat{e_{a+1}}\cdots e_{\sigma(n+m-2-2k)};(e_{a},e_{a+1}))\rangle\\
 & = & (-1)^{m}\sum_{i^{\prime}+j=k+1}\sum_{\sigma^{\prime}\in sh(n-2i^{\prime}-1,m-2j-1)}(-1)^{\sigma^{\prime}}\sum_{{\tau\in sh(i^{\prime},j)\atop (e_{a},e_{a+1})\in\omega}}\\
 &  & \quad(\tilde{\omega_{i^{\prime}}}(e_{\sigma^{\prime}(1)}\cdots)((e_{a},e_{a+1}),f_{\tau(1)}\cdots),\tilde{\eta_{j}}(e_{\sigma^{\prime}(n-2i^{\prime})}\cdots)(f_{\tau(i^{\prime})}\cdots f_{\tau(k)}))\\
 &  & +(-1)^{m}\sum_{i+j^{\prime}=k+1}\sum_{\sigma^{\prime}\in sh(n-2i-1,m-2j^{\prime}-1)}(-1)^{\sigma^{\prime}}\sum_{{\tau\in sh(i^{\prime},j)\atop (e_{a},e_{a+1})\in\eta}}\\
 &  & \quad(\tilde{\omega_{i}}(e_{\sigma^{\prime}(1)}\cdots)(f_{\tau(1)}\cdots),\tilde{\eta_{j^{\prime}}}(e_{\sigma^{\prime}(n-2i)}\cdots)((e_{a},e_{a+1}),f_{\tau(i+1)}\cdots f_{\tau(k)}))\\
 & = & -(\omega\bullet\eta)_{k+1}(e_{1},\cdots,\widehat{e_{a}},\widehat{e_{a+1}},\cdots,e_{n+m-2-2k};(e_{a},e_{a+1}),f_{1},\cdots,f_{k}).
\end{eqnarray*}

$\omega\diamond\eta$ is a cochain in $C^{n+m-2}(L)$ because:

\begin{eqnarray*}
 &  & (\omega\diamond\eta)_{k}(\cdots,e_{a},e_{a+1}\cdots;f_{1}\cdots f_{k})+(\omega\diamond\eta)_{k}(\cdots e_{a+1},e_{a}\cdots;\cdots)\\
 & = & \sum_{i+j=k}\{\sum_{\sigma,a\in\omega,a+1\in\eta}(-1)^{\sigma}\omega_{i+1}(e_{\sigma(1)}\cdots e_{a}\cdots)\circ\eta_{j}(e_{\sigma(n-2i-1)}\cdots e_{a+1}\cdots)\\
 &  & \quad+\sum_{\sigma,a\in\eta,a+1\in\omega}(-1)^{\sigma}\omega_{i+1}(e_{\sigma(1)}\cdots e_{a}\cdots)\circ\eta_{j}(e_{\sigma(n-2i-1)}\cdots e_{a+1}\cdots)\}\\
 &  & +\sum_{i+j=k}\{\sum_{\sigma,a\in\eta,a+1\in\omega}(-1)^{\sigma}\omega_{i+1}(e_{\sigma(1)}\cdots e_{a+1}\cdots)\circ\eta_{j}(e_{\sigma(n-2i-1)}\cdots e_{a}\cdots)\\
 &  & \quad+\sum_{\sigma,a\in\omega,a+1\in\eta}(-1)^{\sigma}\omega_{i+1}(e_{\sigma(1)}\cdots e_{a+1}\cdots)\circ\eta_{j}(e_{\sigma(n-2i-1)}\cdots e_{a}\cdots)\}\\
 &  & +\sum_{i+j=k}\{\sum_{\sigma,a\in\omega,a+1\in\omega}(-1)^{\sigma}\omega_{i+1}(e_{\sigma(1)}\cdots e_{a},e_{a+1}\cdots)\circ\eta_{j}(e_{\sigma(n-2i-1)}\cdots)\\
 &  & \quad+\sum_{\sigma,a\in\omega,a+1\in\omega}(-1)^{\sigma}\omega_{i+1}(e_{\sigma(1)}\cdots e_{a+1},e_{a}\cdots)\circ\eta_{j}(e_{\sigma(n-2i-1)}\cdots)\}\\
 &  & +\sum_{i+j=k}\{\sum_{\sigma,a\in\eta,a+1\in\eta}(-1)^{\sigma}\omega_{i+1}(e_{\sigma(1)}\cdots)\circ\eta_{j}(e_{\sigma(n-2i-1)}\cdots e_{a},e_{a+1}\cdots)\\
 &  & \quad+\sum_{\sigma,a\in\eta,a+1\in\eta}(-1)^{\sigma}\omega_{i+1}(e_{\sigma(1)}\cdots)\circ\eta_{j}(e_{\sigma(n-2i-1)}\cdots e_{a+1},e_{a}\cdots)\}\\
 & = & \sum_{i+j=k}\sum_{{\sigma\in sh(n-2i-2,m-2j)\atop a\in\omega,a+1\in\omega}}(-1)^{\sigma}(-1)\\
 &  & \quad\omega_{i+2}(e_{\sigma(1)}\cdots\widehat{e_{a}},\widehat{e_{a+1}}\cdots e_{\sigma(n-2i-2)};(e_{a},e_{a+1}))\circ\eta_{j}(e_{\sigma(n-2i-1)}\cdots)\\
 &  & +\sum_{i+j=k}\sum_{{\sigma\in sh(n-2i-2,m-2j)\atop a\in\eta,a+1\in\eta}}(-1)^{\sigma}(-1)\\
 &  & \quad\omega_{i+1}(e_{\sigma(1)}\cdots e_{\sigma(n-2i-2)})\circ\eta_{j+1}(e_{\sigma(n-2i-1)}\cdots\widehat{e_{a+1}},\widehat{e_{a}}\cdots;(e_{a},e_{a+1}))\\
 & = & \sum_{i^{\prime}+j=k+1}\sum_{\sigma^{\prime}\in sh(n-2i^{\prime}-2,m-2j)}(-1)^{\sigma^{\prime}}\sum_{{\tau\in sh(j,i^{\prime}-1)\atop (e_{a},e_{a+1})\in\omega}}\\
 &  & \quad\omega_{i^{\prime}+1}(e_{\sigma^{\prime}(1)}\cdots;\eta_{j}(e_{\sigma^{\prime}(n-2i^{\prime}-1)}\cdots;f_{\tau(1)}\cdots),(e_{a},e_{a+1}),f_{\tau(j+1)}\cdots f_{\tau(k)})\\
 &  & +\sum_{i+j^{\prime}=k+1}\sum_{\sigma^{\prime}\in sh(n-2i-2,m-2j^{\prime})}(-1)^{\sigma^{\prime}}\sum_{{\tau\in sh(j,i^{\prime}-1)\atop (e_{a},e_{a+1})\in\eta}}\\
 &  & \quad\omega_{i+1}(e_{\sigma^{\prime}(1)}\cdots;\eta_{j^{\prime}}(e_{\sigma^{\prime}(n-2i-1)}\cdots;(e_{a},e_{a+1}),f_{\tau(1)}\cdots),f_{\tau(j^{\prime})}\cdots f_{\tau(k)})\\
 & = & -(\omega\diamond\eta)_{k+1}(e_{1},\cdots,\widehat{e_{a}},\widehat{e_{a+1}},\cdots,e_{n+m-2-2k};(e_{a},e_{a+1}),f_{1},\cdots,f_{k}).
\end{eqnarray*}

So $\{\omega,\eta\}=\omega\bullet\eta+\omega\diamond\eta-(-1)^{nm}\eta\diamond\omega$
is also a cochain in $C^{n+m-2}(L)$.
\end{proof}
Proof of theorem \ref{thm:poisson bracket Leibniz}:
\begin{proof}
1) By the lemmas above, in order for $\tilde{C}(L)$ to be a graded
Poisson algebra, we need to prove the following:

(1). For any two representable cochains $\omega,\eta$, $\{\omega,\eta\}=-(-1)^{nm}\{\eta,\omega\}$,

(2). For any representable cochains $\omega,\eta,\lambda$, 
\[
\{\omega,\eta\lambda\}=\{\omega,\eta\}\lambda+(-1)^{nm}\eta\{\omega,\lambda\},
\]

(3). The bracket of any two representable cochains is still a representable
cochain, and
\[
\{\omega,\{\eta,\lambda\}=\{\{\omega,\eta\},\lambda\}+(-1)^{nm}\{\eta,\{\omega,\lambda\}\}.
\]

For (1), it suffices to prove $\omega\bullet\eta=-(-1)^{nm}\eta\bullet\omega$. 

$\forall\sigma\in sh(n-2i-1,m-2j-1)$, switching the first $n-2i-1$
arguments with the last $m-2j-1$ arguments results in a sign difference
$(-1)^{(n-1)(m-1)}$, so by definition there is merely a sign difference
between $\omega\bullet\eta$ and $\eta\bullet\omega$ of $(-1)^{n-m+(n-1)(m-1)}=(-1)^{nm+1}$.

Thus (1) is proved.

For (2), we need to prove that $\{\omega,\bullet\}$ is a graded derivative.

\begin{eqnarray*}
 &  & \{\omega,\eta\lambda\}_{k}(e_{1},\cdots,e_{n+m+l-2-2k};f_{1},\cdots,f_{k})\\
 & = & (\omega\bullet\eta\lambda)_{k}(\cdots)+(\omega\diamond\eta\lambda)_{k}(\cdots)+(-1)^{n(m+l)+1}(\eta\lambda\diamond\omega)_{k}(\cdots)
\end{eqnarray*}

We calculate the three parts above respectively:

\begin{eqnarray*}
 &  & (\omega\bullet\eta\lambda)_{k}(e_{1},\cdots,e_{n+m+l-2-2k};f_{1},\cdots,f_{k})\\
 & = & (-1)^{m+l+1}\sum_{{a+b=k\atop {\sigma\in sh(n-2a-1,m+l-2b-1)\atop \tau\in sh(a,b)}}}(\tilde{\omega_{a}}(e_{\sigma(1)},\cdots;f_{\tau(1)},\cdots),\widetilde{(\eta\lambda)_{b}}(e_{\sigma(n-2a)},\cdots;f_{\tau(a+1)},\cdots))\\
 & = & (-1)^{m+l+1}\sum_{{a+b+c=k\atop {\sigma\in sh(n-2a-1,m-2b-1,l-2c)\atop \tau\in sh(a,b,c)}}}(-1)^{\sigma+l}\tilde{(\omega_{a}}(e_{\sigma(1)}\cdots),\tilde{\eta_{b}}(e_{\sigma(n-2a)}\cdots)\lambda_{c}(e_{\sigma(n+m-2a-2b-1)}\cdots))\\
 &  & +(-1)^{m+l+1}\sum_{{a+b+c=k\atop {\sigma\in sh(n-2a-1,m-2b,l-2c-1)\atop \tau\in sh(a,b,c)}}}(-1)^{\sigma}(\tilde{\omega_{a}}(e_{\sigma(1)}\cdots),\eta_{b}(e_{\sigma(n-2a)}\cdots)\tilde{\lambda_{c}}(e_{\sigma(n+m-2a-2b)}\cdots))\\
 & = & \sum_{{a+c=k\atop {\sigma\in sh(n+m-2a-2,l-2c)\atop \tau\in sh(a,c)}}}(-1)^{\sigma}(\omega\bullet\eta)_{a}(e_{\sigma(1)},\cdots)\lambda_{c}(e_{\sigma(n+m-2a-1)},\cdots)\\
 &  & +\sum_{{b+a=k\atop {\sigma\in sh(m-2b,n+l-2a-2)\atop \tau\in sh(b,a)}}}(-1)^{\sigma+(n-1)m}(-1)^{m}\eta_{b}(e_{\sigma(1)},\cdots)(\omega\bullet\lambda)_{a}(e_{\sigma(m-2b+1)},\cdots)\\
 & = & \big((\omega\bullet\eta)\cdot\lambda\big)_{k}(\cdots)+(-1)^{nm}\big(\eta\cdot(\omega\bullet\lambda)\big)_{k}(\cdots)
\end{eqnarray*}

\begin{eqnarray*}
 &  & (\omega\diamond\eta\lambda)_{k}(e_{1},\cdots,e_{n+m+l-2-2k};f_{1},\cdots,f_{k})\\
 & = & \sum_{{a+b=k\atop {\sigma\in sh(n-2a-2,m+l-2b)\atop \tau\in sh(b,a)}}}(-1)^{\sigma}\omega_{a+1}(e_{\sigma(1)}\cdots e_{\sigma(n-2a-2)};(\eta\lambda)_{b}(e_{\sigma(n-2a-1)}\cdots;f_{\tau(1)}\cdots),f_{\tau(b+1)}\cdots)\\
 & = & \sum_{{a+b+c=k\atop {\sigma\in sh(n-2a-2,m-2b,l-2c)\atop \tau\in sh(b,c,a)}}}(-1)^{\sigma}\omega_{a+1}(e_{\sigma(1)}\cdots;\eta_{b}(e_{\sigma(n-2a-1)}\cdots)\lambda_{c}(e_{\sigma(n+m-2a-2b-1)}\cdots),\cdots)\\
 & = & \sum_{{a+b+c=k\atop {\sigma\in sh(n-2a-2,m-2b,l-2c)\atop \tau\in sh(b,a,c)}}}(-1)^{\sigma}\omega_{a+1}(e_{\sigma(1)}\cdots;\eta_{b}(e_{\sigma(n-2a-1)}\cdots),\cdots)\lambda_{c}(e_{\sigma(n+m-2a-2b-1)}\cdots)\\
 &  & +\sum_{{a+b+c=k\atop {\sigma\in sh(m-2b,n-2a-2,l-2c)\atop \tau\in sh(b,c,a)}}}(-1)^{\sigma+nm}\eta_{b}(e_{\sigma(1)}\cdots)\omega_{a+1}(e_{\sigma(m-2b+1)}\cdots;\lambda_{c}(e_{\sigma(n+m-2a-2b-1)}\cdots),\cdots)\\
 & = & \sum_{{a+c=k\atop {\sigma\in sh(n+m-2a-2,l-2c)\atop \tau\in sh(a,c)}}}(-1)^{\sigma}(\omega\diamond\eta)_{a}(e_{\sigma(1)},\cdots;f_{\tau(1)},\cdots)\lambda_{c}(e_{\sigma(n+m-2a-1)},\cdots;f_{\tau(a+1)},\cdots)\\
 &  & +(-1)^{nm}\sum_{{a+b=k\atop {\sigma\in sh(m-2b,n+l-2a-2)\atop \tau\in sh(b,a)}}}(-1)^{\sigma}\eta_{b}(e_{\sigma(1)},\cdots;f_{\tau(1)},\cdots)(\omega\diamond\lambda)_{a}(e_{\sigma(m-2b+1)},\cdots;f_{\tau(b+1)},\cdots)\\
 & = & \big((\omega\diamond\eta)\cdot\lambda\big)_{k}(\cdots)+(-1)^{nm}\big(\eta\cdot(\omega\diamond\lambda)\big)_{k}(\cdots)
\end{eqnarray*}
\begin{eqnarray*}
 &  & (\eta\lambda\diamond\omega)_{k}(e_{1},\cdots,e_{n+m+l-2-2k};f_{1},\cdots,f_{k})\\
 & = & \sum_{{a+c=k\atop {\sigma\in sh(m+l-2a-2,n-2c)\atop \tau\in sh(c,a)}}}(-1)^{\sigma}(\eta\lambda)_{a+1}(e_{\sigma(1)},\cdots;\omega_{c}(e_{\sigma(m+l-2a-1)},\cdots),\cdots)\\
 & = & \sum_{{a+b+c=k\atop {\sigma\in sh(m-2a-2,n-2c,l-2b)\atop \tau\in sh(c,a,b)}}}(-1)^{\sigma+nl}\eta_{a+1}(e_{\sigma(1)}\cdots;\omega_{c}(e_{\sigma(m-2a-1)}\cdots),\cdots)\lambda_{b}(e_{\sigma(n+m-2a-2c-1)}\cdots)\\
 &  & +\sum_{{a+b+c=k\atop {\sigma\in sh(m-2a,l-2b-2,n-2c)\atop \tau\in sh(a,c,b)}}}(-1)^{\sigma}\eta_{a}(e_{\sigma(1)}\cdots)\lambda_{b+1}(e_{\sigma(m-2a+1)}\cdots;\omega_{c}(e_{\sigma(m+l-2a-2b-1)}\cdots),\cdots)\\
 & = & \sum_{{a+b=k\atop {\sigma\in sh(m+n-2a-2,l-2b)\atop \tau\in sh(a,b)}}}(-1)^{\sigma}(-1)^{nl}(\eta\diamond\omega)_{a}(e_{\sigma(1)},\cdots;f_{\tau(1)},\cdots)\lambda_{b}(e_{\sigma(n+m-2a-1)},\cdots;f_{\tau(a+1)},\cdots)\\
 &  & +\sum_{{a+b=k\atop {\sigma\in sh(m-2a,n+l-2b-2)\atop \tau\in sh(a,b)}}}(-1)^{\sigma}\eta_{a}(e_{\sigma(1)},\cdots;f_{\tau(1)},\cdots)\cdot(\lambda\diamond\omega)_{b}(e_{\sigma(m-2a+1)},\cdots;f_{\tau(a+1)},\cdots)\\
 & = & (-1)^{nl}\big((\eta\diamond\omega)\cdot\lambda\big)_{k}(\cdots)+\big(\eta\cdot(\lambda\diamond\omega)\big)_{k}(\cdots)
\end{eqnarray*}

So
\begin{eqnarray*}
 &  & \{\omega,\eta\lambda\}_{k}(e_{1},\cdots,e_{n+m+l-2-2k};f_{1},\cdots,f_{k})\\
 & = & \big((\omega\bullet\eta)\cdot\lambda\big)_{k}(\cdots)+(-1)^{nm}\big(\eta\cdot(\omega\bullet\lambda)\big)_{k}(\cdots)\\
 &  & +\big((\omega\diamond\eta)\cdot\lambda\big)_{k}(\cdots)+(-1)^{nm}\big(\eta\cdot(\omega\diamond\lambda)\big)_{k}(\cdots)\\
 &  & +(-1)^{nm+1}\big((\eta\diamond\omega)\cdot\lambda\big)_{k}(\cdots)+(-1)^{nm}(-1)^{nl+1}\big(\eta\cdot(\lambda\diamond\omega)\big)_{k}(\cdots)\\
 & = & (\{\omega,\eta\}\cdot\lambda)_{k}(\cdots)+(-1)^{nm}(\eta\cdot\{\omega,\lambda\})_{k}(\cdots)
\end{eqnarray*}

$\{\omega,\bullet\}$ is a graded derivative, (2) is proved.

For (3), in order for $\{\omega,\eta\}$ to be a representable cochain,
we need to prove 
\[
\{\omega,\eta\}_{k}(e_{1},\cdots,e_{n+m-2k};f_{1},\cdots,f_{k})=(e_{n+m-2k},\bullet),\ \forall k.
\]
\begin{eqnarray*}
 &  & \{\omega,\eta\}_{k}(e_{1},\cdots,e_{n+m-2-2k};f_{1},\cdots,f_{k})\\
 & = & (-1)^{m+1}\sum_{{a+b=k\atop {\sigma\in sh(n-2a-1,m-2b-1)\atop \tau\in sh(a,b)}}}(-1)^{\sigma}(\tilde{\omega_{a}}(e_{\sigma(1)},\cdots;f_{\tau(1)},\cdots),\tilde{\eta_{b}}(e_{\sigma(n-2a)},\cdots;f_{\tau(a+1)},\cdots))\\
 &  & +\sum_{{a+b=k\atop {\sigma\in sh(n-2a-2,m-2b)\atop \tau\in sh(b,a)}}}(-1)^{\sigma}\omega_{a+1}(e_{\sigma(1)},\cdots;\eta_{b}(e_{\sigma(n-2a-1)},\cdots;f_{\tau(1)},\cdots),f_{\tau(b+1)},\cdots)\\
 &  & +(-1)^{nm+1}\sum_{{a+b=k\atop {\sigma\in sh(m-2a-2,n-2b)\atop \tau\in sh(b,a)}}}(-1)^{\sigma}\eta_{a+1}(e_{\sigma(1)},\cdots;\omega_{b}(e_{\sigma(m-2a-1)},\cdots;f_{\tau(1)},\cdots),f_{\tau(b+1)},\cdots)\\
 & = & (-1)^{m+1}\sum_{{a+b=k\atop {\sigma\in sh(n-2a-2,m-2b-1)\atop \tau\in sh(b,a)}}}(-1)^{\sigma+m+1}\\
 &  & \quad\omega_{a}(e_{\sigma(1)}\cdots e_{\sigma(n-2a-2)},e_{n+m-2-2k},\tilde{\eta_{b}}(e_{\sigma(n-2a-1)}\cdots;f_{\tau(1)}\cdots);f_{\tau(b+1)}\cdots)\\
 &  & +(-1)^{m+1}\sum_{{a+b=k\atop {\sigma\in sh(m-2b-2,n-2a-1)\atop \tau\in sh(a,b)}}}(-1)^{\sigma+(n-1)m}\\
 &  & \quad\eta_{b}(e_{\sigma(1)}\cdots e_{\sigma(m-2b-2)},e_{n+m-2-2k},\tilde{\omega_{a}}(e_{\sigma(m-2b-1)}\cdots;f_{\tau(1)}\cdots);f_{\tau(a+1)}\cdots)\\
 &  & +\{(e_{n+m-2-2k},\bullet)+\sum_{{a+b=k\atop {\sigma\in sh(n-2a-2,m-2b-1)\atop \tau\in sh(b,a)}}}(-1)^{\sigma}\\
 &  & \qquad\omega_{a+1}(e_{\sigma(1)},\cdots;(\tilde{\eta_{b}}(e_{\sigma(n-2a-1)},\cdots;f_{\tau(1)},\cdots),e_{n+m-2-2k}),f_{\tau(b+1)},\cdots)\}\\
 &  & +\{(e_{n+m-2-2k},\bullet)+(-1)^{nm+1}\sum_{{a+b=k\atop {\sigma\in sh(m-2a-2,n-2b-1)\atop \tau\in sh(b,a)}}}(-1)^{\sigma}\\
 &  & \qquad\eta_{a+1}(e_{\sigma(1)},\cdots;(\tilde{\omega_{b}}(e_{\sigma(m-2a-1)},\cdots;f_{\tau(1)},\cdots),e_{n+m-2-2k}),f_{\tau(b+1)},\cdots)\}\\
 & = & (e_{n+m-2-2k},\bullet)\\
 &  & +\sum_{{a+b=k\atop {\sigma\in sh(n-2a-2,m-2b-1)\atop \tau\in sh(b,a)}}}(-1)^{\sigma}\omega_{a}(\cdots,e_{n+m-2-2k},\tilde{\eta_{b}}(\cdots);\cdots)\\
 &  & +\sum_{{a+b=k\atop {\sigma\in sh(m-2b-2,n-2a-1)\atop \tau\in sh(a,b)}}}(-1)^{\sigma+nm+1}\eta_{b}(\cdots,e_{n+m-2-2k},\tilde{\omega_{a}}(\cdots);\cdots)\\
 &  & +\sum_{{a+b=k\atop {\sigma\in sh(n-2a-2,m-2b-1)\atop \tau\in sh(b,a)}}}(-1)^{\sigma+1}\\
 &  & \quad\{\omega_{a}(\cdots,e_{n+m-2-2k},\tilde{\eta_{b}}(\cdots);\cdots)+\omega_{a}(\cdots,\tilde{\eta_{b}}(\cdots),e_{n+m-2-2k};\cdots)\}\\
 &  & +\sum_{{a+b=k\atop {\sigma\in sh(m-2a-2,n-2b-1)\atop \tau\in sh(b,a)}}}(-1)^{\sigma+nm}\\
 &  & \quad\{\eta_{a}(\cdots,e_{n+m-2-2k},\tilde{\omega_{b}}(\cdots);\cdots)+\eta_{a}(\cdots,\tilde{\omega_{b}}(\cdots),e_{n+m-2-2k};\cdots)\}\\
 & = & (e_{n+m-2-2k},\bullet)
\end{eqnarray*}

Thus (3) is proved.
\end{proof}
If $\phi$ is an isomorphism(i.e. the symmetric product of $S^{\bullet}(Z)\otimes L$
is strongly non-degenerate), any $\omega\in C(L)$ is a representable
cochain, so $C(L)=\tilde{C}(L)$ is a graded commutative Poisson algebra.

\section{Derived brackets}

In this section we prove that the Leibniz bracket of a fat Leibniz
algebra is a derived bracket.

First we prove the following:
\begin{prop}
$\{\Theta,\eta\}=-d\eta,\ \forall\eta\in\tilde{C}(L)$.\end{prop}
\begin{proof}
It is obvious that $\Theta$ is a representable cochain.

\begin{eqnarray*}
 &  & (\Theta\bullet\eta)_{k}(e_{1},\cdots,e_{m+1-2k};f_{1},\cdots,f_{k})\\
 & = & (-1)^{m-1}\sum_{\sigma\in sh(2,m-2k-1)}(-1)^{\sigma}(\tilde{\Theta_{0}}(e_{\sigma(1)},e_{\sigma(2)}),\tilde{\eta_{k}}(e_{\sigma(3)}\cdots e_{\sigma(m+1-2k)};\cdots))\\
 &  & +(-1)^{m-1}\sum_{\tau\in sh(1,k-1)}(\tilde{\Theta_{1}}(f_{\tau(1)}),\tilde{\eta_{k-1}}(e_{1},\cdots,e_{m+1-2k};f_{\tau(2)},\cdots,f_{\tau(k)}))\\
 & = & (-1)^{m-1}\sum_{a<b}(-1)^{a+b+1}(e_{a}\circ e_{b},\tilde{\eta_{k}}(e_{1},\cdots,\hat{e_{a}},\cdots,\hat{e_{b}},\cdots e_{m+1-2k};f_{1},\cdots f_{k}))\\
 &  & +(-1)^{m-1}\sum_{i}(-1)(f_{i},\tilde{\eta_{k-1}}(e_{1},\cdots,e_{m+1-2k};f_{1},\cdots,\hat{f_{i}},\cdots,f_{k}))\\
 & = & (-1)^{m}\sum_{a<b}(-1)^{a+b}\eta_{k}(e_{1},\cdots,\hat{e_{a}},\cdots,\hat{e_{b}},\cdots,e_{m+1-2k},e_{a}\circ e_{b};f_{1},\cdots,f_{k})\\
 &  & +(-1)^{m}\sum_{i}\eta_{k-1}(e_{1},\cdots,e_{m+1-2k},f_{i};f_{1},\cdots,\hat{f_{i}},\cdots,f_{k})
\end{eqnarray*}

\begin{eqnarray*}
 &  & (\Theta\diamond\eta)_{k}(e_{1},\cdots,e_{m+1-2k};f_{1},\cdots,f_{k})\\
 & = & \sum_{a}(-1)^{a+1}\Theta_{1}(e_{a})\circ\eta_{k}(e_{1},\cdots,\hat{e_{a}},\cdots,e_{m+1-2k})\\
 & = & \sum_{a}(-1)^{a+1}(-1)(e_{a},\eta_{k}(e_{1},\cdots,\hat{e_{a}},\cdots e_{m+1-2k};f_{1},\cdots,f_{k}))\\
 & = & \sum_{a}(-1)^{a}\rho(e_{a})\eta_{k}(e_{1},\cdots,\hat{e_{a}},\cdots,e_{m+1-2k};f_{1},\cdots,f_{k})
\end{eqnarray*}

\begin{eqnarray*}
 &  & (-1)^{m+1}(\eta\diamond\Theta)_{k}(e_{1},\cdots,e_{m+1-2k};f_{1},\cdots,f_{k})\\
 & = & (-1)^{m+1}\sum_{\sigma\in sh(m-2k-2,3)}(-1)^{\sigma}\\
 &  & \quad\eta_{k+1}(e_{\sigma(1)},\cdots,e_{\sigma(m-2k-2)})\circ\Theta_{0}(e_{\sigma(m-2k-1)},e_{\sigma(m-2k)},e_{\sigma(m-2k+1)})\\
 &  & +(-1)^{m+1}\sum_{a}(-1)^{a+m+1}\eta_{k}(e_{1},\cdots\hat{e_{a}},\cdots,e_{m+1-2k})\circ\Theta_{1}(e_{a})\\
 & = & \sum_{a<b<c}(-1)^{a+b+c+1}\eta_{k+1}(e_{1}\cdots\hat{e_{a}}\cdots\hat{e_{b}}\cdots\hat{e_{c}}\cdots e_{m+1-2k};(e_{a}\circ e_{b},e_{c}),f_{1}\cdots f_{k})\\
 &  & +\sum_{a}(-1)^{a}\sum_{i}\eta_{k}(e_{1},\cdots,\hat{e_{a}},\cdots,e_{m+1-2k};-(e_{a},f_{i}),f_{1},\cdots,\hat{f_{i}},\cdots,f_{k})\\
 & = & \sum_{a<b}(-1)^{a+b}\sum_{b<c<m+2-2k}(-1)^{c}\{\eta_{k}(\cdots,\hat{e_{a}},\cdots,\hat{e_{b}},\cdots,e_{c-1},e_{a}\circ e_{b},e_{c},\cdots;\cdots)\\
 &  & \quad+\eta_{k}(\cdots,\hat{e_{a}},\cdots,\hat{e_{b}},\cdots,e_{c-1},e_{c},e_{a}\circ e_{b},\cdots;\cdots)\}\\
 &  & +\sum_{i,a}(-1)^{a}\{\eta_{k-1}(\cdots e_{a-1},f_{i},e_{a},\cdots;\hat{f_{i}}\cdots)+\eta_{k-1}(\cdots e_{a-1},e_{a},f_{i},\cdots;\hat{f_{i}}\cdots)\}\\
 & = & \sum_{a<b}(-1)^{a+1}\{\eta_{k}(\cdots\hat{e_{a}}\cdots e_{a}\circ e_{b},\cdots;\cdots)+(-1)^{b+m}\eta_{k}(\cdots\hat{e_{a}}\cdots\hat{e_{b}}\cdots e_{a}\circ e_{b};\cdots)\}\\
 &  & -\sum_{i}\eta_{k-1}(f_{i},e_{1},\cdots;\cdots\hat{f_{i}}\cdots)+(-1)^{m+1}\sum_{i}\eta_{k-1}(e_{1}\cdots e_{m+1-2k},f_{i};\cdots\hat{f_{i}}\cdots)
\end{eqnarray*}

The sum of the equations above is
\begin{eqnarray*}
 &  & \{\Theta,\eta\}_{k}(e_{1},\cdots,e_{m+1-2k};f_{1},\cdots,f_{k})\\
 & = & \sum_{a}(-1)^{a}\rho(e_{a})\eta_{k}(e_{1},\cdots,\hat{e_{a}},\cdots,e_{m+1-2k};f_{1},\cdots,f_{k})\\
 &  & +\sum_{a<b}(-1)^{a+1}\eta_{k}(\cdots,\hat{e_{a}},\cdots,e_{a}\circ e_{b},\cdots;\cdots)\\
 &  & -\sum_{i}\eta_{k-1}(f_{i},e_{1},\cdots,e_{m+1-2k};\cdots,\hat{f_{i}},\cdots)\\
 & = & -(d\eta)_{k}(e_{1},\cdots,e_{m+1-2k};f_{1},\cdots,f_{k})
\end{eqnarray*}

The proof is finished.
\end{proof}
Next, we give the definition of fat Leibniz algebras:
\begin{defn}
Given a Leibniz algebra $L$, if the symmetric product $(\cdot,\cdot)$
is non-degenerate, we call $L$ a fat Leibniz algebra.
\end{defn}
The omni Lie algebra $ol(V)=gl(V)\oplus V$ is obviously a fat Leibniz
algebra. And the space of sections of any Courant algebroid is also
a fat Leibniz algebra.

Actually given any Leibniz algebra $L$ with trivial center, there
is associated a fat Leibniz algebra $\tilde{L}$:
\begin{prop}
Suppose $L$ is a Leibniz algebra with trivial center, then $\tilde{L}\triangleq L/K$
is a fat Leibniz algebra, where $K$ is the kernel of the bilinear
product of $L$, i.e. $K=\{k\in L|(k,e)=0,\ \forall e\in L\}$.\end{prop}
\begin{proof}
Since the product of $L$ is invariant: 
\[
\tau(e_{1})(k,e_{2})=(e_{1}\circ k,e_{2})+(k,e_{1}\circ e_{2}),\quad\forall e_{1},e_{2}\in L,\ \forall k\in K,
\]
it follows that 
\[
(e_{1}\circ k,e_{2})=0,\quad\forall e_{2}\in L,
\]
so $e_{1}\circ k\in K$. Furthermore since $e_{1}\circ k+k\circ e_{1}=(k,e_{1})=0$,
so $k\circ e_{1}=-e_{1}\circ k$ is also in $K$. Thus $K$ is an
ideal of $L$. 

The Leibniz bracket of $L$ naturally induces a bracket on $L/K$:
\[
\bar{e}_{1}\circ\bar{e}_{2}\triangleq\overline{e_{1}\circ e_{2}},
\]
where $\bar{e}$ is the equivalent class of $e\in L$ in $L/K$. Suppose
there exists $\bar{k}\in L/K$ such that $(\bar{k},\bar{e})=0,\ \forall\bar{e}\in L/K$,
i.e. $(k,e)\in K,\ \forall e\in L$. Since $(k,e)$ is in the left
center of $L$, $(k,e)\in K$ implies that $(k,e)$ is also in the
right center of $L$. So $(k,e)=0$ by the assumption that the center
of $L$ is trivial. It follows that $k$ itself is in $K$, $\bar{k}=0\in L/K$.
As a result, the bilinear product on $L/K$ is non-degenerate.
\end{proof}
Finally we give the main theorem of this section:
\begin{thm}
\label{thm:derived bracket Leibniz} With the above notations, we
have 
\[
(e_{1}\circ e_{2})^{\flat}=-\{\{\Theta,e_{1}^{\flat}\},e_{2}^{\flat}\}.
\]
 In particular, if $L$ is a fat Leibniz algebra, then the Leibniz
bracket can be represented as a derived bracket:
\[
e_{1}\circ e_{2}=-\{\{\Theta,e_{1}^{\flat}\},e_{2}^{\flat}\}^{\sharp},
\]

where $(\bullet)^{\sharp}:Im((\bullet)^{\flat})\rightarrow L$ is
the (partial) inverse map of $(\bullet)^{\flat}$, i.e.
\[
((\phi)^{\sharp})^{\flat}\triangleq\phi,\quad\forall\phi\in Im((\bullet)^{\flat}).
\]
\end{thm}
\begin{proof}
$\{\Theta,e_{1}^{\flat}\}$ is a 2 cochain:
\begin{eqnarray*}
 &  & \{\Theta,e_{1}^{\flat}\}_{0}(e_{2},e_{3})\\
 & = & \langle\tilde{\Theta_{0}}(e_{2},e_{3}),\tilde{e_{1}^{\flat}}\rangle+\Theta_{1}(e_{2})\circ e_{1}^{\flat}(e_{3})-\Theta_{1}(e_{3})\circ e_{1}^{\flat}(e_{2})\\
 & = & (e_{2}\circ e_{3},e_{1})-(e_{2},(e_{1},e_{3}))+(e_{3},(e_{1},e_{2}))\\
 & = & -(e_{2}\circ e_{1},e_{3})+(e_{3},e_{1}\circ e_{2}+e_{2}\circ e_{1})\\
 & = & (e_{1}\circ e_{2},e_{3})
\end{eqnarray*}

\[
\{\Theta,e_{1}^{\flat}\}_{1}(f)=\langle\tilde{\Theta_{1}}(f),\tilde{e_{1}^{\flat}}\rangle=-(e_{1},f)
\]

We see that $\{\Theta,e_{1}^{\flat}\}$ is a representable cochain.

$\{\{\Theta,e_{1}^{\flat}\},e_{2}^{\flat}\}$ is a 1 cochain:
\begin{eqnarray*}
 &  & \{\{\Theta,e_{1}^{\flat}\},e_{2}^{\flat}\}_{0}(e_{3})\\
 & = & \langle\tilde{\{\Theta,e_{1}^{\flat}\}_{0}}(e_{3}),\tilde{e_{2}^{\flat}}\rangle+\{\Theta,e_{1}^{\flat}\}_{1}\circ e_{2}^{\flat}(e_{3})\\
 & = & (e_{1}\circ e_{3},e_{2})-(e_{1},(e_{2},e_{3}))\\
 & = & -(e_{1}\circ e_{2},e_{3})
\end{eqnarray*}

So $(e_{1}\circ e_{2})^{\flat}(e_{3})=-\{\{\Theta,e_{1}^{\flat}\},e_{2}^{\flat}\}(e_{3})=(e_{1}\circ e_{2},e_{3})$.

The proof is finished.\end{proof}
\begin{rem}
As mentioned in Remark \ref{Remark:standard complex}, $C(L)$ is
isomorphic to the standard complex of the Courant-Dorfman algebra
$S^{\bullet}(Z)\otimes L$. Actually theorem \ref{thm:derived bracket Leibniz}
is true for $e_{1},e_{2}\in S^{\bullet}(Z)\otimes L$. In \cite{Roytenberg},
Roytenberg proved that the Dorfman bracket of a non-degenerate Courant-Dorfman
algebra (i.e. the symmetric product is strongly non-degenerate) is
a derived bracket. Our theorem \ref{thm:derived bracket Leibniz}
can be viewed as a generalization of his result, since the symmetric
product of $S^{\bullet}(Z)\otimes L$ is only non-degenerate, but
not strongly non-degenerate. 

\end{rem}


\begin{thebibliography}{10}
\bibitem{AlekseevXu}Alekseev, Anton and Xu, Ping, $\emph{{\mbox{Derived brackets and Courant algebroids}}}$,
Unpublished manuscript, available at https://www.math.psu.edu/ping/papers.html
(2001).

\bibitem{BenayadiHidri} Benayadi, Saïd and Hidri, Samiha, $\emph{{\mbox{Quadratic Leibniz algebras}}}$,
Journal of Lie Theory 24.3 (2014): 737-759.

\bibitem{Bloh}Bloh, A., $\emph{{\mbox{On a generalization of the concept of Lie algebra}}}$,
Dokl. Akad. Nauk SSSR. Vol. 165. No. 3 (1965), 471-473.

\bibitem{Cai} Cai, Xiongwei, $\emph{{\mbox{H-standard cohomology for Courant-Dorfman algebras and Leibniz algebras}}}$,
arXiv preprint math.RA/1612.05297.

\bibitem{KS}Kosmann-Schwarzbach, Yvette, $\emph{{\mbox{Jacobian quasi-bialgebras and quasi-Poisson Lie groups}}}$,
Mathematical aspects of classical field theory, pages 459\textendash{}489.
Contemp. Math., 132, Amer. Math. Soc., 1992.

\bibitem{KS04}Kosmann-Schwarzbach, Yvette, $\emph{{\mbox{Derived brackets}}}$,
Letters in Mathematical Physics 69.1 (2004): 61-87.

\bibitem{LecomteRoger}Lecomte, P. and Roger, C., $\emph{{\mbox{Modules et cohomologie des bigébres de Lie}}}$,
Comptes rendus Acad. Sci. Paris, 310:405\textendash{}410, 1990.

\bibitem{LiuWX} Liu, Zhang-Ju, Weinstein, Alan and Xu, Ping, $\emph{{\mbox{Manin triples for Lie bialgebroids}}}$,
J. Differential Geom 45.3 (1997): 547-574.

\bibitem{Loday}Loday, J.L., $\emph{{\mbox{Une version non commutative des algèbres de Lie}}}$,
L'Ens. Math. (2), 39 (1993), 269-293.

\bibitem{Roytenberg} Roytenberg, Dmitry, $\emph{{\mbox{Courant\textendash Dorfman algebras and their cohomology}}}$,
Letters in Mathematical Physics 90.1-3 (2009), 311-351.

\bibitem{Roytenberg02}Roytenberg, Dmitry, $\emph{{\mbox{On the structure of graded symplectic supermanifolds and Courant algebroids}}}$,
Contemporary Mathematics 315 (2002): 169-186.\end{thebibliography}
\end{document}